\documentclass[12pt]{amsart}

\usepackage{amssymb,mathrsfs,bm,amsmath,amsthm,color,mathtools,wasysym}
\usepackage{cite,float}
\usepackage{graphicx,caption}
\usepackage{enumerate}
\usepackage[centering]{geometry}
\usepackage[pagebackref]{hyperref}
\allowdisplaybreaks[4]
\theoremstyle{plain}
\newtheorem{thm}{Theorem}[section]
\newtheorem{defn}{Definition}[section]

\newtheorem{prop}[thm]{Proposition}

\newtheorem{lem}[thm]{Lemma}
\theoremstyle{remark}
\newtheorem{rem}{Remark}
\newtheorem{exmp}{Example}
\numberwithin{equation}{section}

\DeclareMathOperator{\hdim}{\dim_H}

\newcommand{\diag}{\mathrm{diag}}
\newcommand{\tba}{T_{\beta_1}}

\newcommand{\tbd}{T_{\beta_d}}
\newcommand{\be}{\bm\epsilon}

\newcommand{\bx}{\mathbf{x}}
\newcommand{\by}{\mathbf{y}}
\newcommand{\bz}{\mathbf{z}}

\newcommand{\qaq}{\mathrm{\quad and\quad}}

\newcommand{\N}{\mathbb N}

\newcommand{\R}{\mathbb R}

\renewcommand{\R}{\mathbb R}
\newcommand{\lm}{\mathcal{L}}
\newcommand{\ck}{\mathcal K}
\newcommand{\ca}{\mathcal A}
\newcommand{\cb}{\mathcal B}
\newcommand{\hc}{\mathcal H_\infty}

\begin{document}
\title{Shrinking parallelepiped targets for $\beta$-dynamical systems}

\author{Yubin He}

\address{Department of Mathematics, Shantou University, Shantou, Guangdong, 515063, China}

\email{ybhe@stu.edu.cn}


\begin{abstract}
	For $ \beta>1 $ let $ T_\beta  $ be the $\beta$-transformation on $ [0,1) $. Let $ \beta_1,\dots,\beta_d>1 $ and let $ \mathcal P=\{P_n\}_{n\ge 1}  $ be a sequence of parallelepipeds in $ [0,1)^d $. Define
	\[W(\mathcal P)=\{\bx\in[0,1)^d:(\tba\times\cdots \times\tbd)^n(\bx)\in P_n\text{ infinitely often}\}.\]
	When each $ P_n $ is a hyperrectangle with sides parallel to the axes, the `rectangle to rectangle' mass transference principle by Wang and Wu [Math. Ann. 381 (2021)] is usually employed to derive the lower bound for $\hdim W(\mathcal P)$, where $\hdim$ denotes the Hausdorff dimension. However, in the case where $ P_n $ is still a hyperrectangle but with rotation, this principle, while still applicable, often fails to yield the desired lower bound. In this paper, we determine the optimal cover of parallelepipeds, thereby obtaining $\hdim W(\mathcal P)$. We also provide several examples to illustrate how the rotations of hyperrectangles affect $\hdim W(\mathcal P)$.
\end{abstract}
\maketitle

\section{Introduction}
The classical theory of Diophantine approximation is concerned with finding good approximations of irrationals. For any irrational $ x\in[0,1] $, if one can find infinitely many rationals $ p/q $ such that $ |x-p/q|<q^{-\tau} $ with $ \tau>2 $, then $ x $ is said to be $\tau$-well approximable. In \cite{HiVe97}, Hill and Velani introduced a dynamical analogue of the classical theory of $\tau$-well approximable numbers. The study of these sets is known as the so-called {\em shrinking target problem}. More precisely, consider a transformation $ T $ on a metric space $ (X,d) $. Let $ \{B_n\}_{n\ge 1} $ be a sequence of balls with radius $ r(B_n)\to 0 $ as $ n\to\infty $. The shrinking target problem concerns the size, especially the Hausdorff dimension, of the set
\[W(T,\{B_n\}_{n\ge 1}):=\{x\in X:T^nx\in B_n\text{ i.o.}\},\]
where `i.o.' stands for {\em infinitely often}. Since its initial introduction, $ W(T,\{B_n\}_{n\ge 1}) $ has been studied intensively in many dynamical systems. See, for example, \cite{AlBa21,BarRa18,BuWa2014,He23,HiVe97,HiVe99,HuWan2022,KLR22,LLVZ22,LWWX14,ShWa13,WaZh21} and reference therein.

 The set $ W(T,\{B_n\}_{n\ge 1}) $ can be thought of as trajectories which hit shrinking  targets $ \{B_n\}_{n\ge 1} $ infinitely often. Naturally, one would like to consider different targets, such as hyperrectangles, rather than just balls. To this end, motivated by the weighted theory of Diophantine approximation, the following set had also been introduced in $\beta$-dynamical system. For $ d\ge 1 $, let $ \beta_1,\dots,\beta_d>1 $ and let $ \mathcal P=\{P_n\}_{n\ge 1} $ be a sequence of parallelepipeds in $ [0,1)^d $. Define
\[W(\mathcal P)=\{\bx\in[0,1)^d:(T_{\beta_1}\times\cdots\times T_{\beta_d})^n(\bx)\in P_n\text{ i.o.}\},\]
where $ T_{\beta_i}\colon[0,1)\to [0,1) $ is given by
\[T_{\beta_i} x=\beta_i x-\lfloor\beta_i x\rfloor.\]
Here $ \lfloor\cdot\rfloor $ denotes the integer part of a real number.
Under the assumption that each $ P_n $ is a hyperrectangle with sides parallel to the axes, the Hausdorff dimension of $ W(\mathcal P) $, denoted by $ \hdim W(\mathcal P) $, was calculated by Li, Liao, Velani and Zorin \cite[Theorem 12]{LLVZ22}. It should be pointed out that their result crucially relies on this assumption. To see this, observe that $ W(\mathcal P) $ can be written as
\[\bigcap_{N=1}^\infty\bigcup_{n=N}^\infty (T_{\beta_1}\times\cdots\times T_{\beta_d})^{-n} P_n.\]
In the presence of such assumption, $ (T_{\beta_1}\times\cdots\times T_{\beta_d})^{-n} P_n $ will be the union of hyperrectangles whose sides are also parallel to the axes. Thus, the `rectangle to rectangle' mass transference principle by Wang and Wu \cite{WaWu21} can be employed to obtain the desired lower bound of $ \hdim W(\mathcal P) $. However, if this assumption is removed, then $ (T_{\beta_1}\times\cdots\times T_{\beta_d})^{-n} P_n $ is in general the union of parallelepipeds, and the mass transference principle, while still applicable, does not work well for this case. The main purpose of this paper is to determine $ \hdim W(\mathcal P) $ without assuming each $ P_n $ is a hyperrectangle. We further show that $ W(\mathcal P) $ has large intersection properties introduced by Falconer~\cite{Falconer1994}, which means that the set $ W(\mathcal P) $ belongs, for some $ 0\le s\le d $, to the class $ \mathscr G^s([0,1]^d) $ of $ G_\delta $-sets, with the property that any countable intersection of bi-Lipschitz images of sets in $ \mathscr G^s([0,1]^d) $ has Hausdorff dimension at least $ s $. In particular, the Hausdorff dimension of $ W(\mathcal P) $ is at least $ s $.

Let
\[f=\diag(\beta_1^{-1},\dots,\beta_d^{-1}).\]
In slightly less rigorous words, the set $(T_{\beta_1}\times\cdots\times T_{\beta_d})^{-n} P_n$ consists of parallelepipeds with the same shape as $f^nP_n$. Note that up to a translation, each $ P_n $ can be uniquely determined by $ d $ column vectors $ \alpha_j^{(n
	)}$. In Lemma \ref{l:paral}, we establish the existence of a rearrangement $f^n\alpha_{i_1}^{(n)},\dots,f^n\alpha_{i_d}^{(n)}$ of $f^n\alpha_1^{(n)},\dots,f^n\alpha_d^{(n)}$, which ensures that upon Garm-Schmidt process, the resulting pairwise orthogonal vectors, denoted by $\gamma_1^{(n)},\dots,\gamma_d^{(n)}$, satisfy the inequality:
\begin{equation}\label{eq:gamma}
	|\gamma_1^{(n)}|\ge\cdots\ge|\gamma_d^{(n)}|>0.
\end{equation}
Most importantly, this yields that up to a multiplicative constant, the optimal cover of $f^nP_n$ is the same as that of the hyperrectangle with sidelengths $|\gamma_1^{(n)}|\ge\cdots\ge|\gamma_d^{(n)}|>0$.
To describe the optimal cover of $(T_{\beta_1}\times\cdots\times T_{\beta_d})^{-n} P_n$, let
\[\ca_n=\{\beta_1^{-n},\dots,\beta_d^{-n},|\gamma_1^{(n)}|,\dots,|\gamma_d^{(n)}|\},\]
and define
\[s_n:=\min_{\tau\in\ca_n}\bigg\{\sum_{i\in\ck_{n,1}(\tau)}1+\sum_{i\notin\ck_{n,1}(\tau)}-\frac{n\log\beta_i}{\log\tau}+\sum_{i\in\ck_{n,2}(\tau)}\bigg(1-\frac{\log|\gamma_i^{(n)}|}{\log\tau}\bigg)\bigg\},\]
where the sets $\ck_{n,1}(\tau)$ and $\ck_{n,2}(\tau)$ are defined as
\begin{equation}\label{eq:ck}
	\ck_{n,1}(\tau):=\{1\le i\le d:\beta_i^{-n}\le \tau\}\qaq \ck_{n,2}(\tau)=\{1\le i\le d:|\gamma_i^{(n)}|\ge \tau\}.
\end{equation}

\begin{thm}\label{t:m}
  Let $\mathcal P=\{P_n\}_{n\ge 1}$ be a sequence of parallelepipeds. For any $n\in\N$, let $\gamma_1^{(n)},\dots,\gamma_d^{(n)}$ be the vectors described in \eqref{eq:gamma}. Then,
	\[\hdim W(\mathcal P)=\limsup_{n\to\infty}s_n=:s^*.\]
	Further, we have $ W(\mathcal P)\in\mathscr G^{s^*}([0,1]^d) $.
\end{thm}

\begin{rem}
	In fact, orthogonalizing the vectors $f^n\alpha_1^{(n)},\dots,f^n\alpha_d^{(n)}$ in different orders will result in different pairwise orthogonal vectors. But not all of them can be well used to illustrate the optimal cover of $f^nP_n$, only those satisfying \eqref{eq:gamma} do. For example, let $P$ be a parallelogram which is determined by two column vectors $\alpha_1=(1,0)^\top$ and $\alpha_2=(m,m)^\top$, $m>1$. Orthogonalizing in the order of $\alpha_1$ and $\alpha_2$ (resp. $\alpha_2$ and $\alpha_1$), we get the orthogonal vectors $\gamma_1=\alpha_1=(1,0)^\top$ and $\gamma_2=(0,m)^\top$ (resp. $\eta_1=\alpha_2=(m,m)^\top$ and $\eta_2=(1/2,-1/2)^\top$). Denote the rectangles determined by $\gamma_1$ and $\gamma_2$ (resp.  $\eta_1$ and $\eta_2$) as $R$ (resp. $\tilde R$). As one can easily see from the following picture, $P$ is contained in the rectangle obtained by scaling $\tilde R$ by a factor of $2$, whereas for $R$ a factor of $m$ is required. Note that $|\gamma_1|<|\gamma_2|$, while $|\eta_1|>|\eta_2|$. This simple example partially inspires us to choose a suitable order to orthogonalize $f^n\alpha_1^{(n)},\dots,f^n\alpha_d^{(n)}$ so that the resulting vectors satisfy \eqref{eq:gamma}, which turns out to be crucial (see Lemma \ref{l:paral} and Equations \eqref{eq:gammak} and \eqref{eq:ak}).
\end{rem}
\begin{figure}[H]
	\centering
	\includegraphics[height=7.5cm, width=15cm]{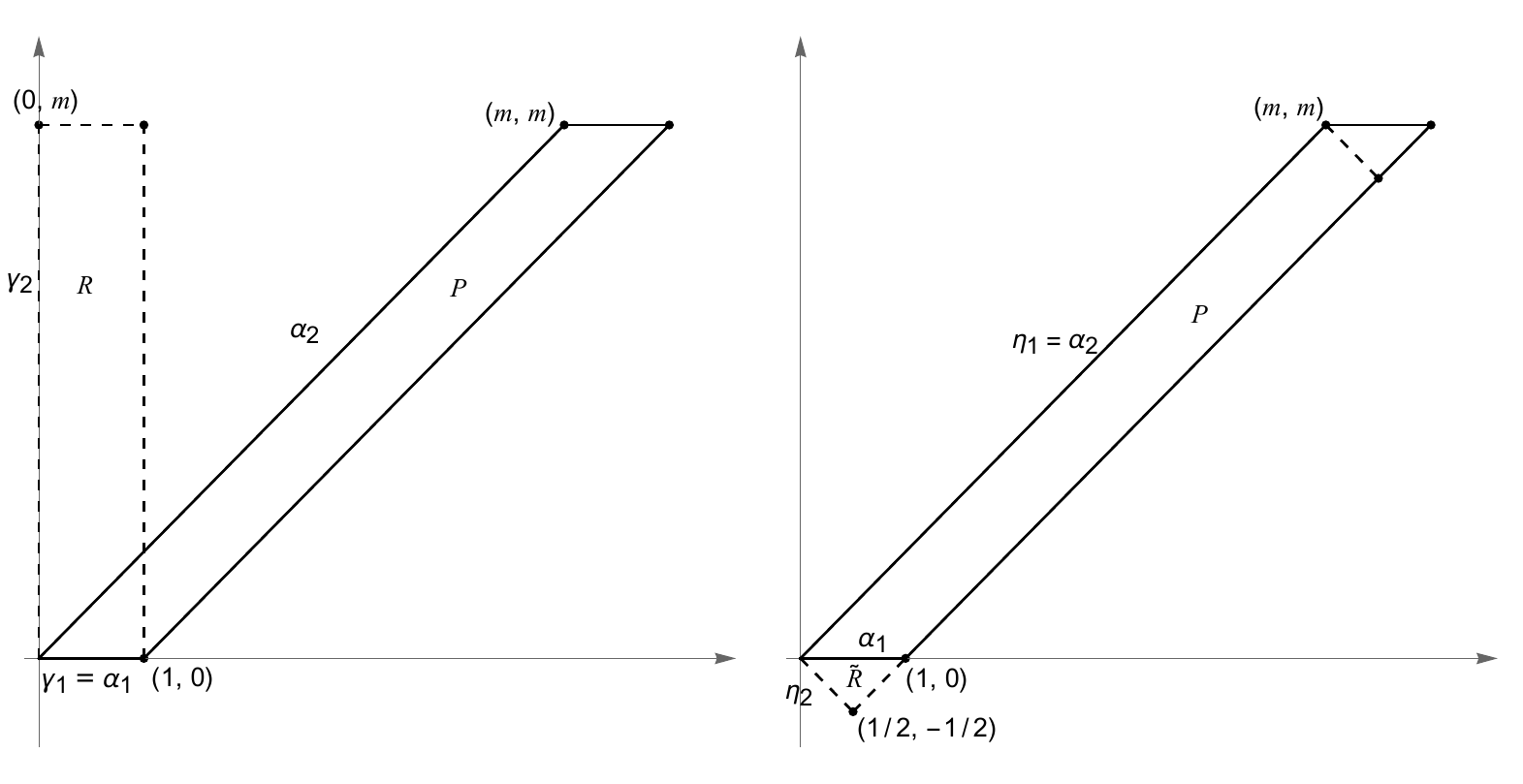}
\end{figure}

\begin{rem}
	Li, Liao, Velani and Zorin \cite[Theorem 12]{LLVZ22} studied an analogous problem, where $P_n$ is restricted to be the following form
	\[P_n=\prod_{i=1}^{d}[-\psi_i(n),\psi_i(n)],\]
	and where $\psi_i$ is a positive function defined on natural numbers, for $1\le i\le d$. They further posed an additional condition that $\limsup_{n\to\infty}-\log\psi_i(n)/n<\infty$ ($1\le i\le d$), as their proof of lower bound for $\hdim W(\mathcal P)$ relies on the `rectangle to rectangle' mass transference principle \cite[Theorem 3.3]{WaWu21}, which demands a similar condition. Their strategy is to investigate the accumulation points of the sequence $\{({-\log\psi_1(n)}/n,\dots,{-\log\psi_d(n)}/n)\}_{n\ge 1}$, and subsequently selecting a suitable accumulation point to construct a Cantor subset of $W(\mathcal P)$, thereby obtaining the lower bound for $\hdim W(\mathcal P)$. However, if $\limsup_{n\to\infty}-\log\psi_i(n)/n=\infty$ for some $i$, they illustrated by an example \cite[\S5.3]{LLVZ22} that this strategy may not achieve the desired lower bound. This problem has been addressed in their forthcoming manuscript, which is currently in progress. We stress that Theorem \ref{t:m} does not pose any similar condition on $P_n$, since we employ a different approach.
\end{rem}

	To gain insight into Theorem \ref{t:m}, we present two examples to illustrate how the rotations of rectangles affect the Hausdorff dimension of $ W(\mathcal P) $.
	\begin{exmp}\label{e:1}
		Let $ \beta_1=2 $ and $ \beta_2=4 $. Let $ \{H_n\}_{n\ge 1} $ be a sequence of rectangles with $ H_n=[0,2^{-n}]\times [0,4^{-n}] $. For a sequence $ \{\theta_n\}_{n\ge 1} $ with $ \theta_n\in [0,\pi/2] $, let
		\begin{equation}\label{eq:pn}
			P_n=R_{\theta_n}H_n+(1/2,1/2),
		\end{equation}
		where $ R_{\theta} $ denotes the counterclockwise rotation by an angle $ \theta $. The translation $(1/2,1/2)$ here is only used to ensure $P_n\subset [0,1)^d$. Suppose that $ \theta_n\equiv\theta $ for all $ n\ge 1 $. For any $ n\ge 1 $ we have
		\[|\gamma_1^{(n)}|=\sqrt{2^{-4n}\cos^2\theta+2^{-6n}\sin^2\theta} \quad\text{and}\quad |\gamma_2^{(n)}|=2^{-6n}/|\gamma_1^{(n)}|.\]
		By Theorem \ref{t:m}, we get
		\[\hdim W(\mathcal P)=\begin{cases}
			5/4&\text{if $ \theta\in[0,\pi/2) $},\\
			1&\text{if $ \theta=\pi/2 $}.
		\end{cases}\]
	\end{exmp}
\begin{exmp}
	Let $ P_n $ be as in \eqref{eq:pn} but with $ \theta_n=\arccos 2^{-an} $ for some $ a> 0 $. Then,
\[|\gamma_1^{(n)}|=\sqrt{2^{-n(4+2a)}+2^{-6n}(1-2^{-na})^2}\quad\text{and}\quad |\gamma_2^{(n)}|=2^{-6n}/|\gamma_1^{(n)}|.\]
By Theorem \ref{t:m}, we get
\[\hdim W(\mathcal P)=\begin{dcases}
	1+\frac{1-a}{4-a}&\text{if $ a\le 1 $},\\
	1&\text{if $ a>1 $}.
\end{dcases}\]
\end{exmp}

	The structure of the paper is as follows. In Section~\ref{s:preliminaries}, we recall several notions and elementary properties of $\beta$-transformation. In Section 3, we estimate the optimal cover of parallelepipeds in terms of Falconer's singular value function. In Section 4, we prove Theorems \ref{t:m}.

\section{$\beta$-transformation}\label{s:preliminaries}
We start with a brief discussion that sums up various fundamental properties of $ \beta $-transformation.

For $ \beta>1 $, let $ T_\beta $ be the $\beta$-transformation on $ [0,1) $.
For any $ n\ge 1 $ and $ x\in[0,1) $, define $ \epsilon_n(x,\beta)=\lfloor \beta T_\beta^{n-1}x\rfloor $. Then, we can write
\[x=\frac{\epsilon_1(x,\beta)}{\beta}+\frac{\epsilon_2(x,\beta)}{\beta^2}+\cdots+\frac{\epsilon_n(x,\beta)}{\beta^n}+\cdots,\]
and we call the sequence
\[\epsilon(x,\beta):=(\epsilon_1(x,\beta),\epsilon_2(x,\beta),\dots)\]
the $\beta$-expansion of $ x $. From the definition of $ T_\beta $, it is clear that, for $ n\ge 1 $, $ \epsilon_n(x,\beta) $ belongs to the alphabet $ \{0,1,\dots,\lceil\beta-1\rceil\} $, where $ \lceil x\rceil $ denotes the smallest integer greater than or equal to $ x $. When $ \beta $ is not an integer, then not all sequences of $ \{0,1,\dots,\lceil\beta-1\rceil\}^\N $ are the $ \beta $-expansion of some $ x\in[0,1) $. This leads to the notion of $\beta$-admissible sequence.

\begin{defn}
	A finite or an infinite sequence $ (\epsilon_1,\epsilon_2,\dots)\in\{0,1,\dots,\lceil\beta-1\rceil\}^\N $ is said to be $\beta$-admissible if there exists an $ x\in[0,1) $ such that the $\beta$-expansion of $ x $ begins with $ (\epsilon_1,\epsilon_2,\dots) $.
\end{defn}

Denote by $ \Sigma_\beta^n $ the collection of all admissible sequences of length $ n $. The following result of R\'enyi~\cite{Renyi1957} implies that the cardinality of $ \Sigma_\beta^n $ is comparable to $ \beta^n $.
\begin{lem}[{\cite[formula 4.9]{Renyi1957}}]\label{l:renyi}
	Let $ \beta>1 $. For any $ n\ge 1 $,
	\[\beta^n\le \# \Sigma_\beta^n\le \frac{\beta^{n+1}}{\beta-1},\]
	where $ \# $ denotes the cardinality of a finite set.
\end{lem}

\begin{defn}
	For any $ \be_n:=(\epsilon_1,\dots,\epsilon_n)\in\Sigma_\beta^n $, we call
	\[I_{n,\beta}(\be_n):=\{x\in[0,1):\epsilon_j(x,\beta)=\epsilon_j,1\le j\le n\}\]
	an $ n $th level cylinder.
\end{defn}
From the definition, it follows that $ T_\beta^n|_{I_{n,\beta}(\be_n)} $ is linear with slope $ \beta^n $, and it maps the cylinder $ I_{n,\beta}(\be_n) $ into $ [0,1) $. If $ \beta $ is not an integer, then the dynamical system $ (T_\beta, [0,1)) $ is not a full shift, and so $ T_\beta^n|_{I_{n,\beta}(\be_n)} $ is not necessary onto. In other words, the length of $ I_{n,\beta}(\be_n) $ may strictly less than $ \beta^{-n} $, which makes describing the dynamical properties of $ T_\beta $ more challenging. To get around this barrier, we need the following notion.
\begin{defn}
	A cylinder $ I_{n,\beta}(\be_n) $ or a sequence $ \be_n\in\Sigma_\beta^n $ is called $\beta$-full if it has maximal length, that is, if
	\[|I_{n,\beta}(\be_n)|=\frac{1}{\beta^n},\]
	where $ |I| $ denotes the diameter of $ I $.
\end{defn}
When there is no risk of ambiguity, we will write full instead of $\beta$-full. The importance of full sequences is based on the fact that the concatenation of any two full sequences is still full.
\begin{prop}[{\cite[Lemma 3.2]{FaWa2012}}]\label{p:concatenation}
	An $ n $th level cylinder $ I_{n,\beta}(\be_n) $ is full if and only if, for any $\beta$-admissible sequence $ \be'_m\in\Sigma_\beta^m $ with $ m\ge 1 $, the concatenation $ \be_n\be_m' $ is still $ \beta $-admissible. Moreover,
	\[|I_{n+m,\beta}(\be_n\be_m')|=|I_{n,\beta}(\be_n)|\cdot|I_{m,\beta}(\be_m')|.\]
	So, for any two full cylinders $ I_{n,\beta}(\be_n), I_{m,\beta}(\be_m') $, the cylinder $ I_{n+m,\beta}(\be_n\be_m') $ is also full.
\end{prop}
For an interval $ I\subset [0,1) $, let $ \Lambda_{\beta}^n(I) $ denote the set of full sequences $ \be_n $ of length $ n $ with $ I_{n,\beta}(\be_n)\subset I $. In particular, if $ I=[0,1) $, then we simply write $ \Lambda_{\beta}^n $ instead of $ \Lambda_{\beta}^n([0,1)) $. For this case, the cardinality of $\Lambda_{\beta}^n$ can be estimated as follows:
\begin{lem}[{ \cite[Lemma 1.1.46]{Li21}}]\label{l:Li}
	Let $\beta>1$ and $n\in \N$.
	\begin{enumerate}[(1)]
		\item If $\beta\in\N$, then
		\[\#\Lambda_{\beta}^n=\beta^n.\]
		\item If $\beta>2$, then
		\[\#\Lambda_{\beta}^n>\frac{\beta-2}{\beta-1}\beta^n.\]
		\item If $1<\beta<2$, then
		\[\#\Lambda_{\beta}^n>\bigg(\prod_{i=1}^{\infty}(1-\beta^{-i})\bigg)\beta^n.\]
	\end{enumerate}
\end{lem}
The general case $I\ne[0,1)$ requires the following technical lemma due to Bugeaud and Wang \cite{BuWa2014}.
\begin{lem}[{\cite[Proposition 4.2]{BuWa2014}}]\label{l:BW}
	Let $\delta>0$. Let $n_0\ge 3$ be an integer such that $(\beta n_0)^{1+\delta}<\beta^{n_0\delta}$. For any interval $I\subset [0,1)$ with $0<|I|<n_0\beta^{-n_0}$, there exists a full cylinder $I_{m,\beta}(\be_m)\subset I$ such that $|I|^{1+\delta}<|I_{m,\beta}(\be_m)|<|I|$.
\end{lem}
 Now, we are ready to tackle with the general case.
\begin{lem}\label{l:inside I}
	Let $ \delta>0 $. Let $ n_0\ge 3 $ be an integer such that $ (\beta n_0)^{1+\delta}<\beta^{n_0\delta} $. Then for any interval $ I $ with $ 0<|I|<n_0\beta^{-n_0} $, there exists a constant $ c_{\beta}>0 $ depending on $\beta$ such that for any $ n\ge -(1+\delta)\log_\beta |I|$,
	\[\#\Lambda_{\beta}^n(I)\ge c_\beta|I|^{1+\delta}\beta^n.\]
\end{lem}
\begin{proof}
	Since $ |I|<n_0\beta^{-n_0} $, by Lemma \ref{l:BW} there exists a full cylinder $ I_{m,\beta}(\be_m) $ satisfying
	\[I_{m,\beta}(\be_m)\subset I\qaq |I|^{1+\delta}< |I_{m,\beta}(\be_m)|=\beta^{-m}<|I|.\]
	For such $m$, we have $ n\ge m $ whenever $ n\ge -(1+\delta)\log_\beta |I| $. By Proposition \ref{p:concatenation}, the concatenation of two full sequences $ \be_{n-m} \in\Lambda_{\beta}^{n-m} $ and $ \be_m $ is still full. Thus,
	\[\#\Lambda_{\beta}^n(I)\ge \#\Lambda_{\beta}^{n-m}\ge c_\beta\beta^{n-m}\ge c_\beta|I|^{1+\delta}\beta^n,\]
	where the constant $c_\beta>0$ depending on $\beta$ is given in Lemma \ref{l:Li}.
\end{proof}
\section{Optimal cover of parallelepipeds}
The proof of Theorem \ref{t:m} relies on finding efficient covering by balls of the $\limsup$ set $W(\mathcal P)$. With this in mind, we need to study the optimal cover of parallelepipeds, which is closely related to its Hausdorff content.

In what follows, for geometric
reasons it will be convenient to equip $\R^d$ with the maximal norm, and thus balls correspond to hypercubes. For any set $E\subset \R^d$, its $s$-dimensional {\em Hausdorff content} is given by
\[\mathcal H^s_\infty(E)=\inf\bigg\{\sum_{i=1}^{\infty}|B_i|^s:E\subset \bigcup_{i=1}^\infty B_i\text{ where $B_i$ are open balls}\bigg\}.\]
  In other words, the optimal cover of a Borel set can be characterized by its Hausdorff content, which is generally estimated by putting measures or mass distributions on it, following the mass distribution principle described below.
\begin{prop}[Mass distribution principle {\cite[Lemma 1.2.8]{BiPe17}}]\label{p:MDP}
	Let $ E $ be a subset of $ \R^d $. If $ E $ supports a strictly positive Borel measure $ \mu $ that satisfies
	\[\mu\big(B(\bx,r)\big)\le cr^s,\]
	for some constant $ 0<c<\infty $ and for every ball $B(\bx,r)$, then $ \mathcal H^s_\infty(E)\ge\mu(E)/c $.
\end{prop}

Following Falconer \cite{Falconer1988}, when $E$ is taken as a hyperrectangle $R$, its Hausdorff content can be expressed as the so-called {\em singular value function}. For a hyperrectangle $ R\subset \R^d $ with sidelengths $ a_1\ge a_2\ge \cdots\ge a_d>0 $ and a parameter $s\in [0,d]$, the singular value function $\varphi^s$ is defined by
\begin{equation}\label{eq:sinR}
	\varphi^s(R)=a_1\cdots a_ma_{m+1}^{s-m},
\end{equation}
where $m=\lfloor s\rfloor$.
%

The next lemma allows us to estimate the Hausdorff content of a Borel set inside a hyperrectangle. Denote the $d$-dimensional Lebesgue measure by $\mathcal L^d$.
\begin{lem}\label{l:content bound of E}
	Let $ E\subset \R^d $ be a bounded Borel set. Assume that there exists a hyperrectangle $R$ with sidelengths $a_1\ge a_2\ge \cdots\ge a_d>0$ such that $E\subset R$ and $\lm^d(E)\ge c\lm^d(R)$ for some $c>0$, then for any $0<s\le d$,
	\[c2^{-d}\varphi^s(R)\le \hc^s(E)\le \varphi^s(R).\]
\end{lem}
\begin{proof}
	The second inequality simply follows from $E\subset R$ and \eqref{eq:sinR}. So, we only need to prove the first one. Let $\nu$ be the normalized Lebesgue measure supported on $E$, i.e.
	\[\nu=\frac{\lm^d|_E}{\lm^d(E)}.\]
	For any $0<s\le d$, let $m=\lfloor s\rfloor$ be the integer part of $s$. Now we estimate the $\nu$-measure of arbitrary ball $B(\bx,r)$ with $r>0$ and $\bx\in E$. The proof is split into two cases.

	\noindent Case 1: $0<r<a_d$. Then,
	\[\begin{split}
		\nu\big(B(\bx,r)\big)&=\frac{\lm^d\big(E\cap B(x,r)\big)}{\lm^d(E)}\le \frac{\lm^d\big(R\cap B(\bx,r)\big)}{c\lm^d(R)}\le\frac{(2r)^d}{ca_1\cdots a_d}\\
		&=\frac{2^dr^s\cdot r^{d-s}}{ca_1\cdots a_ma_{m+1}^{s-m}a_{m+1}^{m+1-s}a_{m+2}\cdots a_d}\le \frac{2^dr^s}{ca_1\cdots a_ma_{m+1}^{s-m}}.
	\end{split}\]

	\noindent Case 2: $a_{i+1}\le r<a_i$ for $1\le i\le d-1$. It follows that
	\[\begin{split}
		\nu\big(B(x,r)\big)&\le \frac{\lm^d\big(R\cap B(x,r)\big)}{c\lm^d(R)}\le\frac{(2r)^i\cdot a_{i+1}\cdots a_d}{ca_1\cdots a_d}=\frac{2^ir^i}{ca_1\cdots a_i}.
	\end{split}\]
	If $i>m=\lfloor s\rfloor$, then the right hand side can be estimated in a way similar to Case 1,
	\[\frac{2^ir^i}{ca_1\dots a_i}=\frac{2^ir^s\cdot r^{i-s}}{ca_1\dots a_ma_{m+1}^{s-m}a_{m+1}^{m+1-s}a_{m+2}\cdots a_i}\le \frac{2^ir^s}{ca_1\dots a_ma_{m+1}^{s-m}}.\]
	If $i\le m=\lfloor s\rfloor$, then $i-s\le 0$, and so
	\[\frac{2^ir^i}{ca_1\dots a_i}=\frac{2^ir^s\cdot r^{i-s}}{ca_1\dots a_i}\le\frac{2^ir^s\cdot a_{i+1}^{i-s}}{ca_1\dots a_i}=\frac{2^ir^s}{ca_1\dots a_ia_{i+1}^{s-i}}\le \frac{2^ir^s}{ca_1\dots a_ma_{m+1}^{s-m}},\]
	where the last inequality follows from the fact that $a_{i+1}\le \cdots\le a_{m+1}$.

	With the estimation given above, by mass distribution principle we have
	\[\hc^s(E)\ge c2^{-d}a_1\dots a_ma_{m+1}^{s-m}=c2^{-d}\varphi^s(R),\]
	as desired.
\end{proof}

By the above lemma, to obtain the optimal cover of a parallelepiped $P$, it suffices to find a suitable hyperrectangle containing it. Since the optimal cover of $P$ does not depend on its location, we assume that one of its vertex lies in the origin. With this assumption, $P$ is uniquely determined by $d$ column vectors, say $\alpha_1,\dots,\alpha_d$. Moreover, we have
\[P=\{x_1\alpha_1+\cdots+x_d\alpha_d:(x_1,\dots,x_d)\in[0,1]^d\}.\]
\begin{lem}\label{l:paral}
	Let $P$ be a parallelepiped given above. There exists a hyperrectangle $R$ such that
	\[P\subset R\qaq \lm^d(P)=2^{-d(d+1)}\lm^d(R).\]
\end{lem}
\begin{proof}
	We will emply Gram-Schmidt process to $\alpha_1,\dots,\alpha_d$ in a proper way to obtain $d$ pairwise orthogonal vectors that yield the desired hyperrectangle.

	First, let $\gamma_1=\alpha_{i_1}$ with $\alpha_{i_1}=\max_{1\le l\le d} |\alpha_l|$. For $1<k\le d$, let $\gamma_k$ be defined inductively as
	\begin{equation}\label{eq:gammak}
		\gamma_k=\alpha_{i_k}-\sum_{j=1}^{k-1}\frac{(\alpha_{i_k},\gamma_j)}{(\gamma_j,\gamma_j)}\gamma_j,
	\end{equation}
	where $\alpha_{i_k}$ is chosen so that
	\begin{equation}\label{eq:ak}
		\bigg|\alpha_{i_k}-\sum_{j=1}^{k-1}\frac{(\alpha_{i_k},\gamma_j)}{(\gamma_j,\gamma_j)}\gamma_j\bigg|=\max_{l\ne i_1,\dots,i_{k-1}}\bigg|\alpha_l-\sum_{j=1}^{k-1}\frac{(\alpha_l,\gamma_j)}{(\gamma_j,\gamma_j)}\gamma_j\bigg|.
	\end{equation}
	This is the standard Gram-Schmidt process and so $\gamma_1,\dots,\gamma_d$ are pairwise orthogonal. Besides,
	\begin{equation}\label{eq:U}
		(\alpha_{i_1},\dots,\alpha_{i_d})=(\gamma_1,\dots,\gamma_d)\begin{pmatrix}
			1 &  -\dfrac{(\alpha_{i_2},\gamma_1)}{(\gamma_1,\gamma_1)} & \cdots & -\dfrac{(\alpha_{i_d},\gamma_1)}{(\gamma_1,\gamma_1)} \\
			0 & 1 & \cdots & -\dfrac{(\alpha_{i_d},\gamma_2)}{(\gamma_2,\gamma_2)} \\
			\vdots & \vdots & \ddots & \vdots \\
			0 & 0 & \cdots & 1
		\end{pmatrix}.
	\end{equation}
	Denote the rightmost upper triangular matrix by $U$. For any $\bx=x_{i_1}\alpha_{i_1}+\cdots+x_{i_d}\alpha_{i_d}\in P$ with $(x_{i_1},\dots,x_{i_d})\in [0,1]^d$, we have
	\[\bx=(\alpha_{i_1},\dots,\alpha_{i_d})\begin{pmatrix}
		x_{i_1} \\
		\vdots \\
		x_{i_d}
	\end{pmatrix}=(\gamma_1,\dots,\gamma_d)U\begin{pmatrix}
		x_{i_1} \\
		\vdots \\
		x_{i_d}
	\end{pmatrix}.\]

	The proof of Lemma \ref{l:paral} will be completed with the help of
	\begin{lem}\label{l:>2}
	The absolute value of each entry of $U$ is not greater than $2$.
	\end{lem}
	\begin{proof}
		For any $1< k\le d$, by the orthogonality of $\gamma_1,\dots,\gamma_{k-1}$,
		\[\begin{split}
			&|\gamma_{k}|^2=(\gamma_k,\gamma_k)=\bigg(\alpha_{i_k}-\sum_{j=1}^{k-1}\frac{(\alpha_{i_k},\gamma_j)}{(\gamma_j,\gamma_j)}\gamma_j,\alpha_{i_k}-\sum_{j=1}^{k-1}\frac{(\alpha_{i_k},\gamma_j)}{(\gamma_j,\gamma_j)}\gamma_j\bigg)\\
			=&\bigg(\alpha_{i_k}-\sum_{j=1}^{k-2}\frac{(\alpha_{i_k},\gamma_j)}{(\gamma_j,\gamma_j)}\gamma_j-\frac{(\alpha_{i_k},\gamma_{k-1})}{(\gamma_{k-1},\gamma_{k-1})}\gamma_{k-1},\alpha_{i_k}-\sum_{j=1}^{k-2}\frac{(\alpha_{i_k},\gamma_j)}{(\gamma_j,\gamma_j)}\gamma_j-\frac{(\alpha_{i_k},\gamma_{k-1})}{(\gamma_{k-1},\gamma_{k-1})}\gamma_{k-1}\bigg)\\
			=&\bigg(\alpha_{i_k}-\sum_{j=1}^{k-2}\frac{(\alpha_{i_k},\gamma_j)}{(\gamma_j,\gamma_j)}\gamma_j,\alpha_{i_k}-\sum_{j=1}^{k-2}\frac{(\alpha_{i_k},\gamma_j)}{(\gamma_j,\gamma_j)}\gamma_j\bigg)-\frac{(\alpha_{i_k},\gamma_{k-1})^2}{(\gamma_{k-1},\gamma_{k-1})}
			\le |\gamma_{k-1}|^2,
		\end{split}\]
		where the last inequality follows from the definition of $\gamma_{k-1}$ (see \eqref{eq:ak}). This gives
		\begin{equation}\label{eq:gam}
			|\gamma_1|\ge|\gamma_2|\ge\cdots\ge|\gamma_d|>0.
		\end{equation}
		By the above inequality and \eqref{eq:ak}, for any $1\le l\le k$, it follows that
		\[\begin{split}
			|\gamma_{l-1}|\ge |\gamma_l|&\ge\bigg|\alpha_{i_k}-\sum_{j=1}^{l-1}\frac{(\alpha_{i_k},\gamma_j)}{(\gamma_j,\gamma_j)}\gamma_j\bigg|\ge\bigg|\frac{(\alpha_{i_k},\gamma_{l-1})}{(\gamma_{l-1},\gamma_{l-1})}\gamma_{l-1}\bigg|-\bigg|\alpha_{i_k}-\sum_{j=1}^{l-2}\frac{(\alpha_{i_k},\gamma_j)}{(\gamma_j,\gamma_j)}\gamma_j\bigg|\\
			&\ge \bigg(\bigg|\frac{(\alpha_{i_k},\gamma_{l-1})}{(\gamma_{l-1},\gamma_{l-1})}\bigg|-1\bigg)|\gamma_{l-1}|,
		\end{split}\]
		which implies that
		\[\bigg|\frac{(\alpha_{i_k},\gamma_{l-1})}{(\gamma_{l-1},\gamma_{l-1})}\bigg|\le 2.\qedhere\]
	\end{proof}

	\textit{Now we proceed to prove Lemma} \ref{l:paral}.

	Let $(U_{i1},\dots,U_{id})$ be the $i$th row of $U$. Since $0\le x_{i_k}\le 1$, by Lemma \ref{l:>2} we have
	\[\left|(U_{i1},\dots,U_{id})\begin{pmatrix}
		x_{i_1} \\
		\vdots \\
		x_{i_d}
	\end{pmatrix}\right|=\bigg|\sum_{k=1}^{d}U_{ik}x_{i_k}\bigg|\le 2^d,\]
	and so
	\begin{equation}\label{eq:defR}
		\bx\in R:=\{x_1\gamma_1+\cdots+x_d\gamma_d:(x_1,\dots,x_d)\in [-2^d,2^d]^d\}.
	\end{equation}
	Therefore, $P\subset R$ which finishes the proof of the first point.

	On the other hand, by an elementary result of linear algebra,
	\begin{align}
		\lm^d(P)&=\text{the absolute value of the determinant $|(\alpha_{i_1},\dots,\alpha_{i_d})|$}\notag\\
		&=\text{the absolute value of the determinant $|(\gamma_1,\dots,\gamma_d)U|$}\notag\\
		&=|\gamma_1|\cdots|\gamma_d|=2^{-d(d+1)}\lm^d(R),\label{eq:measureP=gam}
	\end{align}
	where the third equality follows from the fact that $\gamma_1,\dots,\gamma_d$ are pairwise orthogonal and $U$ is upper triangular with all diagonal entries equal to 1, and the last equality follows from \eqref{eq:defR}.
\end{proof}

\section{Proof of Theorem~\ref{t:m}}
Throughout, we write $a\asymp b$ if $c^{-1}\le a/b\le c$, and $a\lesssim b$ if $a\le cb$ for some unspecified constant $c\ge 1$.
\subsection{Upper bound of $ \hdim W(\mathcal P) $}\label{s:ub}
Obtaining upper estimates for the Hausdorff dimension of a $\limsup$ set is usually straightforward, as it involves a natural covering argument.

For $ 1\le i\le d $, and any $ \be_n^i=(\epsilon_1^i,\dots,\epsilon_n^i)\in\Sigma_{\beta_i}^n $, we always take
\begin{equation}\label{eq:xi}
	z_i^*=\frac{\epsilon_1^i}{\beta_i}+\frac{\epsilon_2^i}{\beta_i^2}+\cdots+\frac{\epsilon_n^i}{\beta_i^n}
\end{equation}
to be the left endpoint of $ I_{n,\beta_i}(\be_n^i) $. Write $ \bz^*=(z_1^*,\dots, z_d^*) $. Then $ W(\mathcal P) $ is contained in the following set
\begin{equation}\label{eq:uplimsup}
	\bigcap_{N=1}^\infty\bigcup_{n=N}^\infty\bigcup_{\be_n^1\in \Sigma_{\beta_1}^n}\cdots\bigcup_{\be_n^d\in\Sigma_{\beta_d}^n}(f^nP_n+\bz^*)=:\bigcap_{N=1}^\infty\bigcup_{n=N}^\infty E_n.
\end{equation}
For any $n\ge 1$, let $f^n\alpha_1^{(n)},\dots,f^n\alpha_d^{(n)}$ be the vectors that determine $f^nP_n$. By Lemma \ref{l:paral} and Equation \eqref{eq:gammak}, there is a hyperrectangle $R_n$ with sidelengths $2^{d+1}|\gamma_1^{(n)}|\ge \cdots\ge 2^{d+1}|\gamma_d^{(n)}|>0$ such that $f^nP_n\subset R_n$.

Recall that $\ca_n=\{\beta_1^{-n},\dots,\beta_d^{-n},|\gamma_1^{(n)}|,\dots,|\gamma_d^{(n)}|\}$, and for any $\tau\in\ca_n$,
\[\ck_{n,1}(\tau):=\{1\le i\le d:\beta_i^{-n}\le \tau\}\qaq \ck_{n,2}(\tau)=\{1\le i\le d:|\gamma_i^{(n)}|\ge \tau\}.\]
Let $\tau\in\ca_n$. We now estimate the number of balls of diameter $\tau$ needed to cover the set $E_n$. We start by covering a fixed parallelepiped $P:=f^nP_n+\bz^*$. In what follows, one can regard $P$ as a hyperrectangle, since $P=f^nP_n+\bz^*\subset R_n+\bz^*$. It is easily verified that we can find a collection $\cb_n(P)$ of balls of diameter $\tau$ that covers $P$ with
\[\#\cb_n(P)\lesssim \prod_{i\in\ck_{n,2}(\tau)}\frac{\big|\gamma_i^{(n)}\big|}{\tau}.\]
Observe that the collection $\cb_n(P)$ will also cover other parallelepipeds contained in $E_n$ along the direction of the $i$th axis with $i\in\ck_{n,1}(\tau)$. Namely, the collection of balls $\cb_n(P)$ simutaneously covers
\[\asymp\prod_{i\in\ck_{n,1}(\tau)}\frac{\tau}{\beta_i^{-n}}\]
parallelepipeds. Since the number of parallelepipeds contained in $E_n$ is less than $\lesssim \beta_1^n\cdots\beta_d^{n}$, one needs at most
\begin{align}
	\lesssim&\bigg(\prod_{i\in\ck_{n,2}(\tau)}\frac{\big|\gamma_i^{(n)}\big|}{\tau}\bigg)\cdot (\beta_1^n\cdots\beta_d^n)\Big/\prod_{i\in\ck_{n,1}(\tau)}\frac{\tau}{\beta_i^{-n}}\notag\\
	=&\prod_{i\in\ck_{n,1}(\tau)}\tau^{-1}\prod_{i\notin\ck_{n,1}(\tau)}\beta_i^{n}\prod_{i\in\ck_{n,2}(\tau)}\frac{\big|\gamma_i^{(n)}\big|}{\tau}\label{eq:numballs}
\end{align}
balls of diameter $\tau$ to cover $E_n$.

Now suppose that $s>s^*=\limsup s_n$, where
\[s_n:=\min_{\tau\in\ca_n}\bigg\{\sum_{i\in\ck_{n,1}(\tau)}1+\sum_{i\notin\ck_{n,1}(\tau)}\frac{n\log\beta_i}{-\log\tau}+\sum_{i\in\ck_{n,2}(\tau)}\bigg(1-\frac{\log|\gamma_i^{(n)}|}{\log\tau}\bigg)\bigg\}.\]
Let $\varepsilon<s-s^*$. For any large $n$, we have $s>s_n+\varepsilon$. Let $\tau_0\in\ca_n$ be such that the minimum in the definition of $s_n$ is attained. In particular,  \eqref{eq:numballs} holds for $\tau_0$. The $s$-volume of the cover of $E_n$ is majorized by
\[\begin{split}
	\lesssim&\Bigg(\prod_{i\in\ck_{n,1}(\tau_0)}\tau_0^{-1}\prod_{i\notin\ck_{n,1}(\tau_0)}\beta_i^{n}\prod_{i\in\ck_{n,2}(\tau_0)}\frac{\big|\gamma_i^{(n)}\big|}{\tau_0}\Bigg)\cdot \tau_0^s\\
	=&\exp\bigg(\sum_{i\in\ck_{n,1}(\tau_0)}-\log\tau_0+\sum_{i\notin\ck_{n,1}(\tau_0)}n\log\beta_i+\sum_{i\in\ck_{n,2}(\tau_0)}(\log|\gamma_{i}^{(n)}|-\log\tau_0)+s\log\tau_0\bigg)\\
	=&\exp\bigg(-\log\tau_0\bigg(\sum_{i\in\ck_{n,1}(\tau_0)}1+\sum_{i\notin\ck_{n,1}(\tau_0)}\frac{n\log\beta_i}{-\log\tau_0}+\sum_{i\in\ck_{n,2}(\tau_0)}\bigg(1-\frac{\log|\gamma_{i}^{(n)}|}{\log\tau_0}\bigg)-s\bigg)\bigg)\\
	=&\exp\big(-\log\tau_0(s_n-s)\big)\le \exp(
	\varepsilon\log\tau_0).
\end{split}\]
Since the elements of $\ca_n$ decay exponentially, the last equation is less than $e^{-n\delta\varepsilon}$ for some $\delta>0$ independent of $n$ and $\varepsilon$. It follows from the definition of $s$-dimensional Hausdorff measure that for any $s$, $\delta>0$ and $\varepsilon$ given above,
\[\mathcal H^s\big(W(\mathcal P)\big)\le\liminf_{N\to\infty}\sum_{n=N}^{\infty}e^{-n\delta\varepsilon}=0.\]
Therefore, $\hdim W(\mathcal P)\le s$. Since this is true for all $s>s^*$, we have
\[\hdim W(\mathcal P)\le s^*=\limsup_{n\to\infty}s_n.\]
\subsection{Lower bound of $ \hdim W(\mathcal P) $}\label{s:lipwd}

The proof crucially relies on the following lemma.
\begin{lem}[{\cite[Corollary 2.6]{He24}}]\label{l:LIP}
	Let $ \{F_n\}_{n\ge 1} $ be a sequence of open sets in $[0,1]^d$ and $ F=\limsup F_n $. Let $ s>0 $. If for any $ 0<t<s $, there exists a constant $ c_t $ such that
	\begin{equation}\label{c:LIP condition}
		\limsup_{n\to\infty} \mathcal H^t_\infty(F_n\cap D)\ge c_t|D|^d
	\end{equation}
	holds for all hypercubes $ D\subset [0,1]^d $, then $ F \in \mathscr G^{s}([0,1]^d) $. In particular, $\hdim F\ge s$.
\end{lem}
\begin{rem}
	A weaker version by Persson and Reeve~\cite[Lemma 2.1]{PerRe2015} also applies to the current proof, but is not adopted here because it results in a more complex proof.
\end{rem}
Let
\[E_n=\bigcup_{\be_n^1\in \Sigma_{\beta_1}^n}\cdots\bigcup_{\be_n^d\in\Sigma_{\beta_d}^n}\big(I_{n,\beta_1}(\be_n^1)\times\cdots\times I_{n,\beta_d}(\be_n^d)\big)\cap(f^nP_n+\bz^*),\]
where $ \bz^*=(z_1^*,\dots,z_d^*) $ is defined as in \eqref{eq:xi}.

\begin{lem}\label{l:LIP of E}
	For any $ 0< t<s^*=\limsup s_n $,
	\[\limsup_{n\to\infty} \mathcal H^t_\infty(E_n\cap D)\apprge |D|^d\]
	holds for all hypercubes $ D\subset [0,1]^d $, where the unspecified constant  depends on $d$ only. Therefore, $W(\mathcal P)\in \mathscr G^{s^*}([0,1]^d)$, and in particular,
	\[\hdim W(\mathcal P)\ge s^*.\]
\end{lem}
\begin{proof}
	Fix $ 0<t<s^* $. Write $\varepsilon=s^*-t$.   By definition, there exist infinitely many $n$ such that
	\begin{equation}\label{eq:sn}
		s_n>t.
	\end{equation}
	In view of Lemma \ref{l:inside I}, let $ D\subset [0,1]^d $ be a hypercube with $ |D|\le n_0\beta_d^{-n_0} $, where $ n_0 $ is an integer such that $ (\beta_i n_0)^{1+\varepsilon/d}<\beta_i^{n_0\varepsilon/d} $ for $1\le i\le d$. Let $ n $ be an integer such that \eqref{eq:sn} holds and for any $1\le i\le d$,
	\begin{equation}\label{eq:condition n}
		n\ge -(1+\varepsilon/d)\log_{\beta_i}|D|\qaq \beta_i^{-n}/2\le |D|^{d+\varepsilon}.
	\end{equation}
	Obviously, there are still infinitely many $n$ that satisfy these conditions. Write $ D=I_1\times \cdots\times I_d $ with $ |I_1|=\cdots=|I_d| $. The first inequality in \eqref{eq:condition n} ensures that Lemma \ref{l:inside I} is applicable to bound $ \#\Lambda_{\beta_i}^n(I_i) $ from below for $ 1\le i\le d $.

	Recall from Lemma \ref{l:paral} and Equation $\eqref{eq:gammak}$ that for any $n\ge 1$, $f^nP_n+z^*$ is contained in some hyperrectangle with sidelengths $2^{d+1}|\gamma_1^{(n)}|\ge \cdots\ge 2^{d+1}|\gamma_d^{(n)}|>0$. For any $n\in\N$ satisfying \eqref{eq:sn} and \eqref{eq:condition n}, define a probability measure $ \mu_n  $ supported on $ E_n\cap D $ by
	\begin{equation}\label{eq:mea mud}
		\mu_n =\sum_{\be_n^1\in \Lambda_{\beta_1}^{n}(I_1)}\cdots\sum_{\be_n^d\in \Lambda_{\beta_d}^n(I_d)}\frac{\nu_{\bz^*}}{\#\Lambda_{\beta_1}^n(I_1)\cdots \#\Lambda_{\beta_d}^n(I_d)},
	\end{equation}
	where $ \nu_{\bz^*} $ is defined by
	\begin{equation}\label{eq:nustar}
		\nu_{\bz^*}:=\frac{\lm^d|_{f^nP_n+\bz^*}}{\lm^d(f^nP_n+\bz^*)}=\frac{\lm^d|_{f^nP_n+\bz^*}}{|\gamma_1^{(n)}|\cdots|\gamma_d^{(n)}|}.
	\end{equation}
	The equality $\lm^d(f^nP_n+\bz^*)=|\gamma_1^{(n)}|\cdots|\gamma_d^{(n)}|$ can be deduced from \eqref{eq:measureP=gam}.

	Let $ \bx\in E_n\cap D $ and $ r>0 $. Suppose that $\bx\in f^nP_n+\by^*\subset E_n\cap D$. Now, we estimate $ \mu_n (B(\bx,r)) $, and the proof is divided into four distinct cases.

	\noindent \textbf{Case 1}: $r\ge|D|$. Clearly, since $t<s\le d$,
	\[\mu_n \big(B(\bx,r)\big)\le 1=\frac{|D|^d}{|D|^d}\le \frac{r^d}{|D|^d}\le \frac{r^t}{|D|^d}.\]

	\noindent \textbf{Case 2}: $ r\le |\gamma_d^{(n)}| $.  Note that in the definition of $\mu_n $, all the cylinders under consideration are full. We see that the ball $ B(\bx,r) $ intersects at most $ 2^d $ parallelepipeds with the form $ f^nP_n+\bz^* $. For any such parallelepiped, by the definition of $\nu_{\bz^*}$ (see \eqref{eq:nustar}) and Lemma \ref{l:inside I} we have
	\begin{align}
		\frac{\mathcal \nu_{\bz^*}\big(B(\bx,r)\big)}{\#\Lambda_{\beta_1}^n(I_1)\cdots \#\Lambda_{\beta_d}^n(I_d)}\notag&\lesssim \frac{1}{\beta_1^n\cdots \beta_d^n|D|^{d+\varepsilon}}\cdot \frac{r^d}{|\gamma_1^{(n)}|\cdots|\gamma_d^{(n)}|}\notag\\
		&= \frac{r^{d-\sum_{i=1}^d(\log\beta_i^n+\log|\gamma_i^{(n)}|)/\log r}}{|D|^{d+\varepsilon}}.\label{eq:case1}
	\end{align}
	Since $f^nP_n+\bz^*$ is contained in some $I_{n,\beta_1}(\be_n^1)\times\cdots\times I_{n,\beta_d}(\be_n^d)$, by a volume argument we have $\sum_{i=1}^d(\log\beta_i^n+\log|\gamma_i^{(n)}|)< 0$. This combined with  $r\le|\gamma_d^{(n)}|<1$ gives
	\begin{align}
		d-\sum_{i=1}^d(\log\beta_i^n+\log|\gamma_i^{(n)}|)/\log r&\ge d-\sum_{i=1}^d(\log\beta_i^n+\log|\gamma_i^{(n)}|)/\log |\gamma_d^{(n)}|\notag\\
		&=\sum_{i=1}^d\frac{n\log\beta_i}{-\log |\gamma_d^{(n)}|}+\sum_{i=1}^d1-\frac{\log|\gamma_i^{(n)}|}{\log|\gamma_d^{(n)}|}.\label{eq:quan1}
	\end{align}
	One can see that the right hand side of \eqref{eq:quan1} is just the one in \eqref{eq:ck} defined by choosing $\tau=|\gamma_d^{(n)}|$, since
	\[\ck_{n,1}(|\gamma_d^{(n)}|)=\emptyset\qaq\ck_{n,2}(|\gamma_d^{(n)}|)=\{1,\dots,d\}.\]
	This means that the quantity in \eqref{eq:quan1} is greater than or equal to $s_n$, and so by \eqref{eq:case1} one has
	\[\mu_n \big(B(\bx,r)\big)\lesssim 2^d\cdot\frac{r^{s_n}}{|D|^{d+\varepsilon}}\lesssim \frac{r^{s_n-\varepsilon}}{|D|^d}\le\frac{r^t}{|D|^d}.\]

	\noindent \textbf{Case 3}: $ \beta_1^{-n}< r\le|D| $. In this case, the ball $B(\bx,r)$ is sufficently large so that for any hyperrectangle $R:=I_{n,\beta_1}(\be_n^1)\times\cdots\times I_{n,\beta_d}(\be_n^d)$,
	\[B(\bx,r)\cap R\ne\emptyset\quad\Longrightarrow\quad R\subset B(\bx,3r).\]
	A simple calculation shows that $B(\bx,r)$ intersects at most $\lesssim r^d\beta_1^n\cdots\beta_d^n$ hyperrectangles with the form $I_{n,\beta_1}(\be_n^1)\times\cdots\times I_{n,\beta_d}(\be_n^d)$. By the definition of $\mu_n $, one has
	\[\begin{split}
		\mu_n \big(B(\bx,r)\big)&\lesssim\frac{1}{\#\Lambda_{\beta_1}^n(I_1)\cdots \#\Lambda_{\beta_d}^n(I_d)}\cdot r^d\beta_1^n\cdots\beta_d^n\\
		&\lesssim\frac{r^d}{|D|^{d+\varepsilon}}\le\frac{r^{d-\varepsilon}}{|D|^d}\le\frac{r^t}{|D|^d}.
	\end{split}\]

	\noindent\textbf{Case 4:} Arrange the elements in $\ca_n$ in non-descending order. Suppose that $\tau_{k+1}\le r<\tau_k$ with $\tau_k$ and $\tau_{k+1}$ are two consecutive terms in $\ca_n$. Let
	\[\ck_{n,1}(\tau_{k+1}):=\{1\le i\le d:\beta_i^{-n}\le \tau_{k+1}\}\qaq \ck_{n,2}(\tau_k)=\{1\le i\le d:|\gamma_i^{(n)}|\ge \tau_k\}.\]
	be defined in the same way as in \eqref{eq:ck}. It is easy to see that $B(\bx,r)$ can intersects at most
	\[\lesssim\prod_{i\in\ck_{n,1}(\tau_{k+1})}r\beta_i^n\]
	parallelepipeds with positive $\mu_n$-measure. Moreover, the $\mu_n $-measure of the intersection of each parallelepiped with $B(\bx,r)$ is mojorized by
	\[\begin{split}
		&\lesssim\frac{1}{|D|^{d+\varepsilon}\beta_1^n\cdots\beta_d^n}\cdot\frac{1}{|\gamma_1^{(n)}|\cdots|\gamma_d^{(n)}|}\cdot \prod_{i\in\ck_{n,2}(\tau_k)}r\cdot\prod_{i\notin\ck_{n,2}(\tau_k)}|\gamma_i^{(n)}|\\
		&=\frac{1}{|D|^{d+\varepsilon}\beta_1^n\cdots\beta_d^n}\cdot \prod_{i\in\ck_{n,2}(\tau_k)}\frac{r}{|\gamma_i^{(n)}|}.
	\end{split}\]
	Therefore,
	\[\begin{split}
		\mu_n \big(B(\bx,r)\big)&\lesssim \bigg(\prod_{i\in\ck_{n,1}(\tau_{k+1})}r\beta_i^n\bigg)\cdot \bigg(\frac{1}{|D|^{d+\varepsilon}\beta_1^n\cdots\beta_d^n}\cdot \prod_{i\in\ck_{n,2}(\tau_k)}\frac{r}{|\gamma_i^{(n)}|}\bigg)\\
		&=\frac{1}{|D|^{d+\varepsilon}}\cdot\prod_{i\in\ck_{n,1}(\tau_{k+1})}r\cdot \prod_{i\notin\ck_{n,1}(\tau_{k+1})}\beta_i^{-n}\cdot \prod_{i\in\ck_{n,2}(\tau_k)}\frac{r}{|\gamma_i^{(n)}|}\\
		&=\frac{r^{s(r)}}{|D|^{d+\varepsilon}},
	\end{split}\]
	where
	\[s(r)=\sum_{i\in\ck_{n,1}(\tau_{k+1})}1+\sum_{i\notin\ck_{n,1}(\tau_{k+1})}\frac{-n\log\beta_i}{\log r}+\sum_{i\in\ck_{n,2}(\tau_k)}\bigg(1-\frac{\log|\gamma_i^{(n)}|}{\log r}\bigg).\]
	Clearly, as a function of $r$, $s(r)$ is monotonic on the interval $[\tau_{k+1},\tau_k]$. So the minimal value attains when $r=\tau_{k+1}$ or $\tau_k$. First, suppose that the minimum is attained at $r=\tau_k$. If $\ck_{n,1}(\tau_k)=\ck_{n,1}(\tau_{k+1})$, then there is nothing to be proved. So we may assume that $\ck_{n,1}(\tau_k)\ne\ck_{n,1}(\tau_{k+1})$. Since $\ck_{n,1}(\tau_{k+1})\subsetneq\ck_{n,1}(\tau_k)$, one can see that $\tau_k=\beta_j^{-n}$ for some $j$ and
	\[\ck_{n,1}(\tau_k)=\ck_{n,1}(\tau_{k+1})\cup\{j\}.\]
	It follows that
	\[\sum_{i\in\ck_{n,1}(\tau_{k+1})}1+\sum_{i\notin\ck_{n,1}(\tau_{k+1})}\frac{-n\log\beta_i}{\log \tau_k}=\sum_{i\in\ck_{n,1}(\tau_{k})}1+\sum_{i\notin\ck_{n,1}(\tau_{k})}\frac{-n\log\beta_i}{\log \tau_k},\]
	which implies that
	\[s(r)\ge s_n.\]
	By a similar argument one still have $s(r)\ge s_n$ if the minimum is attained at $r=\tau_{k+1}$. Therefore,
	\[\mu_n \big(B(\bx,r)\big)\lesssim\frac{r^{s_n}}{|D|^{d+\varepsilon}}\le \frac{r^{s_n-\varepsilon}}{|D|^d}\le\frac{r^t}{|D|^d}.\]

	Summarizing the estimates of the $\mu_n $-measures of arbitrarily balls presented in Cases 1--4, we get
	\[\mu_n \big(B(\bx,r)\big)\lesssim\frac{r^t}{|D|^d}\quad\text{for all }r>0,\]
	where the unspecified constant does not depend on $D$. Finally, by mass distribution principle,
	\[\mathcal H^t_{\infty}(E_n\cap D)\apprge|D|^d.\]
	This is true for infinitely many $n$, and the proof is completed.
\end{proof}

\subsection*{Acknowledgements}
This work was supported by NSFC (Nos. 1240010704). The author would like to thank Prof. Lingmin Liao for bringing this problem to his attention. Also the author is grateful to the anonymous referees for
their patience and efforts to improve the quality of the manuscript.

\end{document}